\documentclass[12pt]{amsart}

\usepackage{amsmath,amssymb,amsthm,xypic,epsfig }
\usepackage{CJK}
\usepackage{enumerate}
\usepackage[dvips]{color}
\usepackage{version}
\usepackage{enumitem}
\usepackage{moreenum}

\DeclareFontFamily{OT1}{rsfs}{}
\DeclareFontShape{OT1}{rsfs}{n}{it}{<-> rsfs10}{}
\DeclareMathAlphabet{\curly}{OT1}{rsfs}{n}{it}

\newtheorem{Thm}{Theorem}[section]
\newtheorem{Lem}[Thm]{Lemma}
\newtheorem{Fac}[Thm]{Fact}
\newtheorem{Cor}[Thm]{Corollary}

\newtheorem{Prop}[Thm]{Proposition}

\newtheorem{``Conj"}[Thm]{``Conjecture"}

\newtheorem{Claim}[Thm]{Claim}

\theoremstyle{remark}
\newtheorem{Rem}[Thm]{Remark}
\newtheorem{Ex}[Thm]{Example}
\theoremstyle{definition}
\newtheorem{Def}[Thm]{Definition}

\newtheorem{Step}{Step}

\newtheorem*{ack}{Acknowledgments}

\newcommand{\SL}{\mathop{\mathrm{SL}}\nolimits}

\newcommand{\iso}{\xrightarrow{   \,\smash{\raisebox{-0.40ex}{\ensuremath{\scriptstyle\simeq}}}\,}}

\excludeversion{memo}

\begin{document}

\title[Tropical geometric moduli compactification]
{Tropical Geometric Compactification of Moduli, I - $M_g$ case - }
\author{Yuji Odaka}

\dedicatory{To the memory of Kentaro Nagao}

\maketitle
\thispagestyle{empty}

\begin{abstract}
We compactify the classical moduli variety of compact Riemann surfaces 
by attaching moduli of (metrized) \textit{graphs} as boundary. 
The compactifications do \textit{not} admit the structure of varieties and 
patch together to form a big connected moduli space 
in which $\sqcup_{g} M_{g}$ is open dense. 

The metrized graphs, which are often studied as ``tropical curves'', 
are obtained as  
Gromov-Hausdorff collapse by fixing diameters of the hyperbolic metrics of 
the Riemann surfaces. 
This phenomenon can be also seen as an archemidean analogue of 
the tropicalization of Berkovich analytification of $M_{g}$ \cite{ACP}. 

\end{abstract}


\tableofcontents


\section{Introduction}

Let us recall that the moduli space of smooth projective curves admits a 
``canonical" modular compactification 
constructed in 
Deligne-Mumford \cite{DM} first as an algebraic stack 
$\overline{\mathcal{M}_{g}}^{{\rm DM}}$. 
\footnote{Here we put the superscript ``DM'', often omitted in the literaturs, 
to clearly distinguish from 
the compactifications we introduce in this paper. } 
Later on, the moduli stack was 
proved to have a coarse projective variety which is normal and of dimension 
$3g-3$ 
\cite[Especially, III]{KM}, \cite{Gie}, \cite{Mum77}. 

The boundary of the compactification still parametrizes geometric 
objects which are certain nodal curves 
called ``stable curves'' characterized by the GIT stability 
(\cite{Gie}, \cite{Mum77}) or by the K-stability (\cite[4.1]{Od2}, also cf. 
\cite{Mum77}, \cite{Od1}, \cite[\S 7]{LW}). Hence the 
GIT construction (\cite{Mum65}) applies (\cite{Gie}, \cite{Mum77}) while it also fits to 
more general moduli existence conjecture for K-(semi)stable polarized varieties 
(``K-moduli" cf., \cite{Od4}). 

In this paper, we introduce a pair of new 
compactifications of $M_{g}$ 
which are 
\textit{no longer varieties} 
but compact Hausdorff toplogical spaces. 
In the first compactification which we denote as $\overline{M_{g}}^{\rm GH}$, 
the boundaries parametrize the Gromov-Hausdorff limits of 
compact Riemann surfaces with rescaled Poincar\'e (i.e., K\"ahler-Einstein) 
metrics with diameter $1$, 
which we identify 
as certain graphs (Theorem \ref{GH.curves}). 
Hence we would like to call the compactification $\overline{M_{g}}^{\rm GH}$ 
\textit{Gromov-Hausdorff compactification}. 

In the second compactifications of $M_{g}$, we further encode 
some non-negative integer weights on the vertices of the limit graphs. 
We call the metrized graphs with such weights, \textit{weighted metrized graphs}. 
The class of our limits graph is very close to what has been studied as 
``(stable) tropical curves" 
in the literatures (e.g., \cite{BMV,Cap,MZ,CHMR}). 
Our point is that we can construct a refined compactification of $M_g$ than 
$\overline{M_{g}}^{\rm GH}$ by encoding 
the weights. 
The obtained 
compactifications will be called ``{\textit{tropical geometric compactifications}". 
We chose the term because the boundaries 
coincides with the moduli spaces of such tropical curves, which are also studied 
in the literatures (e.g., \cite{BMV,Cap,MZ,CHMR} again), 
while we also avoided the 
term ``\textit{tropical compactification}'' already used by J.~Tevelev 
whose context is very different, namely, 
the problem of  compactifying subvarieties of a torus in a  toric variety (cf.,  \cite{Tev}). 

Let us explain the backgrounds by 
discussing a broader picture for moduli spaces of 
more general varieties. There are two major backgrounds for this work, which we recall now: 

\begin{enumerate}
\item \label{SYZ}
The current extensive approach to the Strominger-Yau-Zaslow mirror symmetry conjecture (\cite{SYZ}). Indeed, 
conjectures of Gross-Wilson \cite[\S6]{GW}, Todorov, 
Kontsevich-Soibelman (\cite{KS}) (cf., e.g., the survey on 
the Gross-Siebert program\cite{Gross}) speculates 
certain families of Calabi-Yau varieties with its Ricci-flat K\"ahler metrics 
collapse to integral affine manifolds with singularities 
in the Gromov-Hausdorff sense, which are recently often regarded as some 
tropical version of Calabi-Yau varieties. 

\item \label{DS}
The algebraicity of \textit {non-collapsed} Gromov-Hausdorff limits of K\"ahler-
Einstein manifolds  (\cite{DS}), its applications to moduli of Fano varieties 
(\cite{Spo,OSS,Od4}), later followed by 
(\cite{SSY,LWX,Od5}). 
\end{enumerate}

There is a similarity between the above two i.e. (\ref{SYZ}) and (\ref{DS})
as the first i.e. (\ref{SYZ}) is in particular showing that the collapsed Gromov-Hausdorff 
limits of K\"ahler-Einstein manifolds are ``tropical \textit{algebraic}" objects 
while the second (\ref{DS}) is showing that the 
non-collapsed limits of K\"ahler-Einstein (Fano) manifolds are \textit{algebro-geometric}  
objects i.e., varieties. 

For moduli spaces of Fano manifolds, 
which we discussed in 
(cf., \cite{DS,OSS,Od4}, \cite{SSY,LWX,Od5}), 
the two kinds of the compactifications 
\begin{enumerate}
    \item[($\alpha$)] the Gromov-Hausdorff metric compactification of 
the moduli space of K\"ahler-Einstein manifolds with the rescaled K\"ahler-Einstein 
metrics with fixed diameters (our $\overline{M_{g}}^{\rm GH}$ and $\overline{M_{g}}^{\rm T}$ to be introduced 
in this paper are on this side) which is closer to the spirit of \eqref{SYZ}
and 
\item[($\beta$)]\label{beta} algebro-geometric compactified moduli of K-stable varieties, 
e.g. $\overline{M_{g}}^{\rm DM}$ as in \eqref{DS}
\end{enumerate}
essentially 
coincide because of the non-collapsing of the metrics. 
However they ``look'' completely different 
in the non-Fano case due to collapse of the K\"ahler-Einstein metrics 
as we show in the present series of papers. 
Indeed, the author believes that the Gromov-Hausdorff compactification while 
fixing the \textit{volume} (rather than the diameter), 
if it exists in an appropriate sense, should be closer in spirit to ($\beta$). 
Nevertheless, as we observe in the case of $M_{g}$ in this paper, we believe that the two series of 
compactifications $(\alpha)$ and $(\beta)$ must be deeply connected in general. 

In the present paper, first we start with the classification of all the possible Gromov-Hausdorff limits 
of the compact Riemann surfaces 
with K\"{a}hler-Einstein metrics of diameters $1$. 
Then using the classification, we construct 
the compactifications and proceed to analyze their structures. 

Our connection between classical algebro-geometric compactifications and tropical moduli spaces 
can be seen as an archimedean 
analogue of the \textit{tropicalization} (\textit{skeleton}) of non-archimedean 
analytification 
of the moduli varieties which is recently studied in \cite{ACP}. We discuss this analogy towards the end of the subsection \ref{GH.collapse.curves}. 

Another interesting point of our compactifications $\overline{M_g}^{\rm GH}$, is that they naturally patch together to form a big (infinite dimensional) \textit{conneted} moduli space 
in which $M_{g}$ are open subsets for \textit{all} $g$. We will call it \textit{infinite join} and denotes it as 
$\overline{M_{\infty}}^{\rm GH}$. 

It would be interesting to pursue this line of research for moduli varieties of other polarized varieties. 
For instance, the author conjectures that the moduli schemes of 
smooth canonical models, again with the rescaled K\"ahler-Einstein metrics of diameters $1$, 
are also precompact for 
Gromov-Hausdorff distance and the corresponding 
collapses will be dual intersection complexes of KSBA semi-log-canonical models 
in certain generalized sense. 
Such speculation is inspired by 
the recent Koll\'ar-Shepherd-Barron-Alexeev  
(KSBA) compactification (cf., e.g., the survey \cite{Kol}) 
and the observation that it is a moduli scheme of K-stable varieties 
(\cite{Od1,Od3}, also \cite{BG}). 
Another interesting case would be those of 
polarized K3 surfaces whose 
Gromov-Hausdorff collapse in the maximally degenerate case have been studied 
in (\cite{GW}, 
\cite{KS}, \cite{GTZ}), which is in 
\cite{OO}. 

Throughout this article, we work over the complex number field $\mathbb{C}$ 
unless otherwise stated. 


\begin{ack}
The first version of this paper appeared in June, 2014 (arXiv:1406.7772) and this is a revised exposition 
of the \textit{former half}, i.e. the $M_g$ case, 
of the original preprint. The companion paper \cite{Od.Ag} is a 
revision of the \textit{latter half}, i.e. the $A_g$ part of arXiv:1406.7772, 
together which included later developments. 

The author would like to thank Radu Laza, Valentino Tosatti, Shouhei Honda, Daisuke Kishimoto, Takeo Nishinou, Takao Yamaguchi for helpful discussions and 
Simon Donaldson, Kei Irie, Hiroshi Iritani, Nariya Kawazumi, Richard Thomas for their helpful comments and 
interests which encouraged me. The author also would like to thank Lionel Lang for teaching 
him his paper \cite{LL} (see Remark \ref{LL}) 
on June of 2015, and thank also the anonymous referee and Yoshiki Oshima 
who helped the author to improve the presentation recently. 

This paper and its companion paper \cite{Od.Ag} 
are dedicated to fifteen years memory of \textit{Kentaro Nagao}. 
Looking back, I can never stop deeply thanking Nagao-san for all the inspirations 
from the 
beginning and the warm friendliness. 
I hope he would be delighted again. 

\end{ack}


\section{Gromov-Hausdorff compactification of $M_{g}$}

\subsection{Precompactness}

For each compact Riemann surface of genus $g(\ge 2)$, 
we put \textit{rescale} of the K\"ahler-Einstein metric with 
the \textit{diameter} $1$. 
\footnote{Readers will find later that this specific constant $1$ does not 
have any specific meaning as we only meant to fix it, 
so we can rather set it to be any fixed positive constant.}
Recall that the 
K\"ahler-Einstein metric is nothing but the famous Poincar\'e metric 
in this case. 
The first point we should clarify is the precompactness of $M_{g}$ with the associated Gromov-Hausdorff distance 
(for its definition we refer to e.g. \cite[Chapter 7]{BBI}) 
on it. We denote the Gromov-Hausdorff distance as $d_{\rm dGH}$. 
Recall that the precompactness of a subset 
of the space of all compact metric spaces means its closure 
with respect to the Gromov-Hausdorff topology is compact. 
During the process of degenerations i.e., 
going to boundary of $M_{g}$, 
the curvature tends to $- \infty$, so we can \textit{not} apply 
the Gromov's precompactness theorem \cite{Grom} 
in our situation. 
Instead we can apply the following theorem of Shioya \cite{Shi} and the Gauss-Bonnet theorem to prove it. 

\begin{Thm}[{\cite[Theorem 1.1]{Shi}}]\label{Shioya.theorem}
For two fixed positive real numbers $D>0$ and $c>0$, consider 
the set $S(D,c)$ of closed $2$-dimensional Riemannian manifolds $(R,d)$ with 
\begin{enumerate}
\item
the diameter ${\rm diam}(d)<D$ 
\item 
and the total absolute curvature $\int_{R} |K_{(R,d)}|{\rm vol}(R)<c$ where $K_{(R,d)}$ 
and ${\rm vol}(R)$ denotes the Gaussian curvature and the volume form with respect to the 
metric $d$. 
\end{enumerate}
Then the set $S(D,c)$ is precompact with respect to the associated Gromov-Hausdorff distance. 
\end{Thm}

By applying the above theorem, we get the following desired precompactness. 

\begin{Cor}\label{Mg.precompact}
$(M_{g},d_{\rm dGH})$ is precompact. 
\end{Cor}
\begin{proof}[First proof]
It directly follows from the Shioya's theorem above (\ref{Shioya.theorem}) since 
our total absolute curvature is constant due to the Gauss-Bonnet theorem. 
\end{proof}

We include another proof of Corollary \ref{Mg.precompact} in the next section, 
in which we also classify all the Gromov-Hausdorff limits. 


\subsection{Gromov-Hausdorff collapse of Riemann surfaces}\label{GH.collapse.curves}

Before stating a theorem, we precisely fix some graph theoretic terminology we use in this paper. 

\begin{Def}
In the present paper, a \textit{metrized (finite) graph} means a finite connected non-directed graph with finite positive lengths attached to 
all edges. It is not necessarily simple, i.e., loops and several edges with the same ends are allowed. 
A \textit{contraction} of a finite graph is a graph which 
can be obtained from the original graph by 
contracting some of its edges. 
\end{Def}

The main result of this section is the following theorem, 
which implies the precompactness of $M_g$ and 
also classify all the possible Gromov-Hausdorff limits of compact hyperbolic surfaces while fixing 
their diameters. 

\begin{Thm}\label{GH.curves}
Let $\{R_{i}\}_{i\in \mathbb{Z}_{>0}}$ be an arbitrary sequence of compact Riemann surfaces of fixed genus $g\geq 2$. 
Suppose $\{(R_{i}, \frac{d_{{\rm KE}}}{\rm diam(R_{i})})\}_{i}$ converges in the Gromov-Hausdorff sense. 
Here $d_{\rm KE}$ denotes the Poincar\'e metric
\footnote{i.e., the hyperbolic metric which is also a K\"ahler-Einstein metric, 
hence the notation} 
on each $R_{i}$ and its diameter is ${\rm diam}(R_{i})$. 

Then the limit is the metric space associated to either 
\begin{enumerate}
\item \label{conv.to.graph} a metrized graph of diameter $1$ or 
\item \label{conv.to.surf} a compact Riemann surface of genus $g$. 
\end{enumerate}

Assume furthermore that the sequence $R_{i}$ converges to 
$[R_{\infty}]\in \overline{M}_{g}^{\rm DM}$ (which can be always be 
achieved by passing to a subsequence 
since $\overline{M}_{g}^{\rm DM}$ is compact). 
Then if $[R_{\infty}]\in M_{g}$ we are in case (\ref{conv.to.surf}) 
and $R_{i}$ converges in the Gromov-Hausdorff sense 
to the metric space underlying $R_{\infty}$; if, on the other hand, 
$[R_{\infty}]\not\in M_{g}$ then we are in case (\ref{conv.to.graph}) 
and the $R_{i}$ converges to the metric 
space underlying a metrized graph whose underlying graph 
is a contraction of the dual graph of $R_{\infty}$. 

Conversely, 
any metrized graph with diameter $1$ 
whose underlying graph is a contraction of some (possibly $0$) 
edges of the dual graph of a
 stable curve of genus $g$, can occur in this way (\ref{conv.to.graph}). 
\end{Thm}

\begin{proof}
We fix a reference compact Riemann surface $S$ and regard the Teichmuller space $T_{g}$ as the 
set of marked compact Riemann surfaces $[\phi\colon S\iso R]$ where we only care 
of the isotopy type of $\phi$. 

First we briefly recall the basic of the pair-of-pants decomposition of $S$, 
which we abbreviate as pants decomposition from now on for short, 
and later we will explain how to apply it. 
$$
S=\bigcup_{0\leq a\leq g-2} P_a
$$
with the associated simple closed boundary geodesics $s_{1},\cdots,s_{3g-3}$. 
Then in turn it naturally induces the corresponding pants decompositions of $R$
$$
R=\bigcup_{0\leq a\leq g-2} P_a(R)
$$
for all elements $[\phi\colon S\iso R]$ of $T_{g}$ since we can take 
simple closed boundary geodesics in the corresponding homology classes. 
The associated simple closed boundary geodesics $\{s_{j}(R)\}_{j}$ of $R$  gives 
the (real analytic) Fenchel-Nielsen coordinates on it 
$$(l_{1},\cdots,l_{3g-3};\theta_{1},\cdots,\theta_{3g-3})
\colon T_{g}\cong \mathbb{R}_{>0}^{3g-3}\times (\mathbb{R}/2\pi \mathbb{Z})
^{3g-3}, $$
where $l_{j}$ is the length of $s_{j}$ and $\theta_{j}$ is corresponding 
twist parameters (cf., \cite{IT}). 
Then the following well-known theorem is due to L.~Bers. 

\begin{Fac}[{\cite{Bers}}]\label{Bers}
Fix a positive integer $g\ge 2$. Then there is a uniform constant $C_g$ 
such that for an arbitrary compact hyperbolic Riemann surface $R$, there is a pant 
decomposition whose corresponding lengths $l_{i}$ of 
simple closed geodesics satisfy $l_{i}<C_g$. 
\end{Fac}
We now argue as follows. 
Suppose we are given a 
sequence $\{R_{i}\}_{i\in \mathbb{Z}_{>0}}$ 
of compact Riemann surfaces of the fixed genus $g\ge 2$, 
as in the statement of Theorem \ref{GH.curves}. 
We replace it by 
its certain subsequence, after several steps as follows. 
Firstly, due to the compactness of the Deligne-Mumford compactification 
$\overline{M_g}^{\rm DM}$, 
we can replace the sequence $\{R_i\}$ by subsequence, if necessary, 
to ensure the existence of a limit of $[R_i]$ inside $\overline{M_g}^{\rm DM}$. 
By applying the Bers' theorem \ref{Bers}, for each $i$, we have a pants 
decomposition 
satisfying the assertion of Theorem \ref{Bers}, i.e., all the lengths of the 
corresponding
 simple closed geodesics are less than a uniform constant $C_g$. 
 For each $R_{i}$, we fix such a pants decomposition from now on. 
On the other hand, note that for each pants decomposition there is 
a corresponding graph whose vertices are (pair of) pants while edges are common geodesics is 
$3$-regular with $2g-2$ vertices. We call this graph 
the combinatorial type of the pant decomposition. See for instance 
\cite[around Definition 1.5]
{Ham11} for the details. 
The number of edges of such a graph 
is $3g-3$ so  obviously there is only 
a finite possibilities for such graphs. Hence, there is only a 
finite possibilities of combinatorial 
type of pants decomposition. 
Therefore, by passing to an appropriate subsequence of $\{R_{i}\}$ again, 
if necessary, we can and do assume the combinatorial type of the pants decompositions we took, which satisfies \ref{Bers}, 
stays fixed. After that, 
by further passing to an appropriate subsequence of $\{R_{i}\}$ again, 
if necessary, we can and do assume moreover that 
$\lim_{i\to \infty} l_{j}(R_{i})=L_{j}$ for some constants 
$L_{j}\in [0,\infty)$ for each $j$. 

The simple geodesics $s_{j}(R_i)$ of $R_{i}$ with $L_{j}=0$ are 
representing the vanishing cycles, i.e., all the cycles that 
shrink to nodal singularities of the corresponding limit 
in the Deligne-Mumford compactification $\overline{M_{g}}^{\rm DM}$. 
We make the following claim, although the author believes this 
has been known to or expected by the experts. 

\begin{Claim}\label{pinching}
There is an index $j$ with $L_{j}=0$, 
if and only if the diameter of 
the \textit{non-rescaled} hyperbolic metrics (i.e., with constant 
Gaussian curvature $-1$) 
tends to $+\infty$. 
This is also equivalent to that the 
limit of the sequence $[R_{i}]$ does not belong to $M_{g}$. 

Otherwise, passing to a subsequence, 
the Gromov-Hausdorff limit $R_{\infty}$ of $\{R_{i}\}_{i}$ 
exists as a compact Riemann surface 
of the same genus $g$. 
\end{Claim}

\begin{proof}[proof of Claim \ref{pinching}]

If all the $L_{j}$ are non-zero, then the compactness of 
$$\{(l_{1},\cdots,l_{3g-3}; 
\theta_{1},\cdots,\theta_{3g-3}) \mid 
L_i-\epsilon \le l_{i}\le C_{g} \text{ for } 1\le \forall i\le 3g-3 \} 
\subset T_{g}$$ for small enough positive real number $\epsilon$ 
straightforwardly implies that the corresponding points 
$[R_i]\in T_{g}$ converge inside $T_{g}$. 
Here, we recall a standard fact that 
the underlying topological surface with the Poincar\'e metric 
to each $[R]\in M_{g}$ (or its rescales with diameters $1$) 
varies contiuously for $R$ with respect to the 
Gromov-Hausdorff topology. If follows, for instance, 
from the interpretation of the family as a 
family of the quotients of the upper half plane by 
continuously deforming Fuchsian subgroup of $PSL(2,\mathbb{R})$. 
(The isomorphic class of the Fuchsian group is not changed, 
as it is the isomorphic class of the 
fundamental group of genus $g$ compact Riemann surface.) 
Or it also follows from the implicit function theorem applied to the 
constancy of the Gaussian curvature. 
Hence, in particular, the diameters 
of the (non-rescaled, original) Poincar\'e metrics of $R_i$ are bounded and 
the Gromov-Hausdorff limit of 
$R_i$ with the rescaled Poincar\'e metric 
is still a compact Riemann surface of genus $g$. 
On the other hand, if $L_{j}=0$ for at least one index $j$, then 
the famous collar theorem \cite{Ke} applies and directly shows that 
for each $i$ there is a cylinder (called ``collar") 
inside $R_i$, including the closed geodesic $l_j$, 
whose diameter tends to $+\infty$. 
We end the proof of the Claim \ref{pinching}. 
\end{proof}

From now on, we assume these equivalent conditions are satisfied i.e., 
$[R_{\infty}]\notin M_{g}$. 
Otherwise, the subsequence converges to a compact Riemann surface (i.e., ``does not degenerate''), 
which corresponds to the case $(ii)$ of 
Theorem \ref{GH.curves}. This is again because of the 
continuity of the surfaces with the 
rescaled Poincar\'e metrics 
parametrized by $M_{g}$ 
with respect to the Gromov-Hausdorff topology. 

Let us denote the diameter of the Poincar\'e (hyperbolic) 
metric $d_{\rm KE}$ of $R_{i}$ as $d_{i}$. 
Then recall that what we are analysing 
is the metric behaviour of $(R_{i},\frac{d_{\rm KE}}{d_{i}})$ and 
we wish to determine its Gromov-Hausdorff limit. 
For that, we analyze the behaviour of the pant 
$(P_{a}(R_{i}),\frac{d_{\rm KE}}{d_{i}})$ in this proof. 
We denote the three boundary geodesics of the pants as $s_{b}(P_a) (b=1,2,3)$, 
or $s_{b}(R_i;P_a) (b=1,2,3)$ for precision, 
which may partially be identified in the Riemann surface $R_i$ i.e., 
e.g. $s_1(R_i;P_{a})=s_2(R_i;P_{a})$ can be possible. 
From now on, whenever the context is clear, 
we sometimes omit $R_i$ and simply denote the 
pants of $R_i$ as $P_a$, not $P_a(R_i)$ and its boundary geodesics 
$s_{b}(P_a) (b=1,2,3)$ rather than 
$s_{b}(R_i;P_a) (b=1,2,3)$. 

Let us recall a standard fact in the Teichmuller theorey (cf., \cite[Chapter 3, 
\S 1.5, \S2]{IT}) which claims 
that the pant $P_{a}(R_{i})$ can be cut and separated into two isometric hyperbolic hexagons $Q_{a}(R_{i})$ and $Q_{a}'(R_{i})$ 
canonically by geodesics which connect different boundary geodesics of the pant 
$P_a(R_{i})$. Let us also recall from \cite[Chapter 3, \S1.5, \S2]{IT} that the 
interior part of the hyperbolic hexagons $Q_a(R_{i})$, with its hyperbolic metric, 
can be regarded as an open subset of a unit disc with the hyperbolic metric $d_{\rm KE}$, 
in a unique way up to the isometry group of the disk i.e., ${\it PGL}(2,\mathbb{R})$. 
We denote the center of the unit disc as $p$. 

Let us call the $3$ boundaries of the hexagon 
which were originally part of the boundaries of the pant $P_{a}$ as 
\textit{boundary geodesics}. 
In any case, the important invariants are the lengths of the $3$ boundary geodesics 
which are half of the lengths of 
the boundary geodesics $s_{b}(R_i;P_{a}) (b=1,2,3)$ of the original pant $P_{a}$. 
Indeed, it is a well-known fact that biholomorphic type of $Q_{a}$ (so also for $P_{a}$) is determined by the 
lengths of the three boundary geodesics (cf., e.g., \cite{IT}). 
We now study 
the Gromov-Hausdorff limit of the hyperbolic hexagon $Q_{a}$ 
while fixing diameters. 
Then, recall from the Claim \ref{pinching}, it follows that 
$d_{\rm KE}(p,s_{b}(R_{i};P_a))\rightarrow +\infty$ for $i\rightarrow \infty$ if and only if 
the corresponding boundary geodesic $s_{b}(R_{i};P_{a})$ shrinks 
i.e., ${\rm length}(s_{b}(R_{i};P_{a}))\to 0$ for $i\to \infty$. 

To each $P_{a}$, 
we associate a tree $\Gamma_{a}$, just as a combinatorial graph, with 
\begin{itemize}
\item the vertex set $V(\Gamma_{a}):=\{v_{a}\}\sqcup 
\{w_{b}  \mid  s_{b}(R_{i};P_{a}) \text{ shrinks} \}$ and 
\item the edge set $E(\Gamma_{a}):=\{\overline{v_{a}w_{b}}\mid s_{b}(R_{i};P_{a}) 
\text{ shrinks} \}$. 
\end{itemize}

Denote the diameter of the hyperbolic hexagon $Q_{a}(R_{i})$ with respect to Poincar\'e metric as $d_{i}(a)$. 
(Recall that the diameter of whole $R_i$ is $d_i$. )
We analyze the asymptotic behaviour of $(R_i, \frac{d_{\rm KE}}{d_i})$ by further 
``decomposing" into that of $Q_a(R_i)$ as above. 

First we fix 
a constant $0<\epsilon\ll 1$ so that the sequence of the half pant 
$\{Q_{a}(R_{i})\}_{i}$ satisfies that the disk $D(p,(1-\epsilon))$ with center $p$ 
and radius $(1-\epsilon)$ 
contains all non-shrinking boundary geodesics of $Q_{a}(R_{i})$. Then 
thinking of the distance between each point in 
$(Q_{a}(R_{i})\cap D(p,(1-\epsilon))$ and $p$, we straightforwardly obtain that 
the diameter of 
$\{(Q_{a}(R_{i})\cap D(p,(1-\epsilon)),d_{\rm KE})\}_{i}$ is bounded above 
by $C_{\epsilon}$. 
On the other hand, the diameters of the collar neighborhoods 
of shrinking boundary geodesics tends to 
$+\infty$ by the collar theorem \cite{Ke}. 
Hence, we have that 
\begin{Claim}[Diverging hyperbolic hexagon]
\label{diverge}
$d_{i}(a)\to \infty$ for $i\to \infty$ if and only if there is an index $b$ 
with ${\rm length}(s_{b}(R_{i};P_{a}))\to 0$ for $i\to \infty$. 
\end{Claim}
\begin{Claim}[Limit of hyperbolic hexagon, I]\label{local.GH.lim}
If we consider the sequence $(Q_{a}(R_{i}),\frac{d_{\rm KE}}{d_{i}(a)})$ for $i=1,2,\cdots$, it 
has the Gromov-Hausdorff limit as a metrized tree $\Gamma_{a}$ in the 
case when ${\rm length}(s_{b}(R_{i};P_{a}))\to 0$ for some $b$ 
when $i\to \infty$. Otherwise its Gromov-Hausdorff limit is still some hyperbolic hexagon. 
\end{Claim}

Next, we compare the diameters of each hyperbolic hexagon $Q_{a}(R_{i})$ and the 
whole Riemann surface $R_{i}$. 

\begin{Claim}
[Diameters comparison]
\label{diam.compare}
\begin{enumerate}
\item \label{i}
For any $i$ there is at least one $Q_{a}(R_{i})$ (or equivalently, 
its index $a$) such that 
\begin{equation}\label{i.ineq}
d_{i}\le 12(g-1)d_{i}(a).
\end{equation}
\item \label{ii}
Suppose that an index $a$ satisfies that 
$d_{i}(a)\to \infty$ 
when $i\to \infty$. Then, 
for any $a$ and large enough $i$, we have 
\begin{equation}\label{ii.ineq}
\frac{d_{i}(a)}{2}\le d_{i}.
\end{equation}
\end{enumerate}
\end{Claim}
\begin{proof}[Proof of Claim \ref{diam.compare}]
The second assertion \eqref{ii} easily follows from the definition. Indeed, 
it can be proven as follows. 
First we can assume $d_{i}(a)$ is the length of a geodesic 
$\gamma\colon [0,1]\to R_{i}$ 
connecting two points $\gamma(0), \gamma(1)$ in the union of the 
boundary geodesics. Then its midpoint $\gamma(\frac{1}{2})$ and 
one of the endpoints, say $\gamma(1)$, 
of the geodesic has the same distance in whole $R_{i}$ i.e., 
after gluing the boundary geodesics. Hence 
\eqref{ii} follows. 

Our first assertion \eqref{i} is proved 
as follows. Take a shortest geodesic 
$\delta\colon [0,1]\to R_{i}$ connecting two points in $R_{i}$ with 
${\rm length}(\delta)={\rm diam}(R_{i})$. 
An elementary observation shows that 
\begin{align}
\label{3}
{\rm length}({\rm Im}(\delta)\cap Q_{a}(R_{i}))\le 3d_{i}(a)\\ 
\label{4}
{\rm length}({\rm Im}(\delta)\cap Q'_{a}(R_{i}))\le 3d_{i}(a). 
\end{align}
Indeed, if we write 
$$I_{a}:=\{t\in [0,1]\mid \delta(t)\in Q_{a}\}=[\alpha_{1},
\alpha_{2}]\sqcup \cdots \sqcup [\alpha_{2m-1},\alpha_{2m}],$$
with $0\le \alpha_{1}\le \alpha_{2}\le \cdots \alpha_{2m}$, 
then note that $\delta(\alpha_{2})$ and $\delta(\alpha_{2m-1})$ are  
connected by a geodesic of length at most 
$d_{i}(a)$, by the definition of $d_{i}(a)$. Since $\delta$ is 
taken to be a shortest geodesic, 
$\sum_{1\le k<m}{\rm length}(\delta([\alpha_{2k-1},\alpha_{2k}]))
\le d_{i}(a)$ which gives our desired estimate \eqref{3}, and also \eqref{4} 
similarly. Hence, by summing up, we obtain 
$d_{i}\le 6\sum_{a}d_{i}(a).$ Since $\#\{a\}=2(g-1)$, we obtain the desired inequality \eqref{ii.ineq}. 
\end{proof}

From the Claims \ref{diverge} 
and \ref{diam.compare} \eqref{i},\eqref{ii} we have that 
$d_{i}\to \infty$ if and only if there is some $P_{a}$ with 
$d_{i}(a)\to \infty$. 
Also it follows from the Claim \ref{diam.compare}, if $P_{a}$ satisfies that 
for some $b$ 
${\rm length}(s_{b}(R_{i};P_{a}))\to 0$ for $i\to \infty$, 
by further passing to a subsequence we can assume that $R_{i}$ satisfies that 
$$\frac{d_{i}(a)}{2}\le d_{i}\le 12(g-1)d_{i}(a),$$ 
for a fixed $a$, say $a=1$. On the other hand,   
$$\frac{d_{i}(a)}{2}\le d_{i}$$ holds for any $a$. 
Hence, combining Claim \ref{local.GH.lim} and Claim \ref{diam.compare}, 
we have that 
\begin{Claim}
[Limit of hyperbolic hexagon, II] 
\label{local.GH.lim2} 
Under our assumption that 
$[R_{\infty}]\notin M_{g}$, 
if we consider the sequence $(Q_{a}(R_{i}),\frac{d_{\rm KE}}{d_{i}})$ for 
$i\to \infty$, it converges in the Gromov-Hausdorff sense to 
either a metrized tree $\Gamma$ or a point. 
\end{Claim}
From the above claim \ref{local.GH.lim2}, 
it follows 
that the global Gromov-Hausdorff limit of $(R_{i},\frac{d_{\rm KE}}{d_{i}})$ is 
a metrized graph which is obtained by gluing all $\Gamma_{a}$ at $w_{b}$'s whose corresponding 
boundary geodesics $s_{j}$ are the same in the whole Riemann surface 
$R_{i}$. The resulting graph is 
either the dual graph of the 
corresponding stable curve $R_{\infty}$ or 
a graph obtained from the dual graph after contracting several edges to 
points. (We simply call such procedure a \textit{contraction} of a graph 
in this paper.) 

\vspace{3mm}

Now let us move on to the proof of the converse direction 
(the last paragraph of the statements of Theorem \ref{GH.curves}). 
That is, starting from an arbitrary finite metrized graph $\Gamma$ of diameter $1$ 
which satisfies the assumption of the last paragraph of Theorem \ref{GH.curves}, 
we wish to prove there is a 
sequence of compact Riemann surfaces $R_{i} (i=1,2,\cdots)$ of genus $g$ such that 
$\Gamma$ is the Gromov-Hausdorff limit of 
$(R_{i},\frac{d_{\rm KE}}{d_{i}})$ i.e., the rescaled 
Poincar\'e metrics of diameter $1$. 

We fix an arbitrary stable curve $R$ whose dual graph contracts to the 
underlying graph of $\Gamma$. Such $R$ exists due to our assumption on $\Gamma$. 
Then take a smooth point in 
each of the irreducible components of $R$ and denote them by $p_{i}$. 
Here the index $i$ corresponds to each irreducible component. 
We take a semi-universal deformation of $R$ as 
$\{R_{\vec{t}}\}_{\vec{t}\in U}$ with 
an open neighborhood $U' \subset \mathbb{C}^{3g-3}$ of $\vec{0}$, 
satisfying $R_{\vec{0}}=R$ and take $p_{i,\vec{t}}$ of $R_{\vec{t}}$ with 
$p_{i,\vec{0}}=p_{i}$ 
which is continuous with respect to $\vec{t}$. 
From here, we use a smaller open neighborhood of $\vec{0}$ denoted by $U\subset U'$ 
with $\bar{U}\subset U'$. Note that there is a discriminant locus $D\subset U$ 
such that $\vec{t}\notin D$ if and only if $R_{\vec{t}}$ is smooth. 
We fix a uniform pants decomposition of $R_{\vec{t}}$ so that 
the nodes $x_{k}$ of $R$ are shrunk dividing geodesics $s_{k}$ 
of the decomposition. 
For each node $x_{k}$ of $R$ connecting the irreducible components 
including $p_{i}$ and $p_{j}$, 
there is a corresponding  shortest geodesic 
$\gamma_{k,\vec{t}}$ 
connecting $p_{i,\vec{t}}$ and $p_{j,\vec{t}}$ if $R_{\vec{t}}$ is smooth 
which intersects with $s_{k}$.

Recall that 
there is a standard submersive holomorphic map 
$\phi=\{\phi_{k}\}_{k}\colon U\to 
\prod_{k}{\rm Kur}(x_{k})$, where ${\rm Kur}(x_{k})$ stands for 
the 
Kuranishi space 
underlying a semi-universal deformation of the node singularity $x_{k}$, 
and $\phi_{k}$ is induced by restricting the deformation of $R$ to 
a neighborhood of each node $x_{k}$. In this case, ${\rm Kur}(x_{k})$ 
can be regarded as an open neighborhood of $0$ in $\mathbb{C}$ and the 
discriminant locus $D$ is the divisor $\cup_{k}\phi_{k}^{-1}(0)$. 
For the proof of the fact that $\phi$ is submersive, i.e., its differential 
$d\phi$ is surjective, see \cite[Proposition 1.5]{DM}. 
Since the distance of $p_{i},p_{j}$ for $i\neq j$ in $R$ with respect to 
the hyperbolic metric is $+\infty$ (i.e., not defined as a real number), 
for a sequence $\{\vec{t}_{m}\}_{m=1,2,\cdots}\subset U\setminus D$, 
$${\rm length}(\gamma_{k,\vec{t}_{m}}; R_{\vec{t}_{m}})\to +\infty$$ 
for $m\to \infty$ if and only if $\phi_{k}(\vec{t}_{m})\to \vec{0}$. 

On the side of $\Gamma$, 
for each node $x_{k}$ of $R$, also an 
edge $\gamma_{k}$ of $\Gamma$ corresponds, which may be 
possibly contracted to a point. If it is contracted, we regard it as an  
edge of length $0$. 

From the above discussions with the surjectivity of $\phi$, 
for large enough positive integers $m\gg 1$, there is $\vec{t_{m}}\in U\setminus D$ 
\begin{align}
\label{edge1}
&{\rm length}(\gamma_{k,\vec{t_{m}}}; R_{\vec{t_{m}}})
=m\cdot {\rm length}(\gamma_{k};\Gamma) \text{ if } \gamma_{k} 
\text{ is 
not contracted in }\Gamma  \\ 
\label{edge2}
&{\rm length}(\gamma_{k,\vec{t_{m}}}; R_{\vec{t_{m}}})
=\sqrt{m} \text{ if } \gamma_{k} \text{ is contracted in }\Gamma. 
\end{align}
Then, the above taken sequence of smooth 
compact Riemann surfaces 
$\{R_{\vec{t}_{m}}\}_{m}$ with the rescaled 
Poincar\'e metric converges to a metrized graph and from 
\eqref{edge1} and \eqref{edge2}, the limit metrized graph coincides with 
$\Gamma$. We complete the proof of the last paragraph of 
Theorem \ref{GH.curves}. 
\end{proof}

\begin{Rem}
An additional comment about the above proof is in order. 
Since for each $\Gamma$, 
the above construction of sequence $R_{\vec{t}_{m}}$ depends on 
our arbitrary choice $R$ (indeed $\lim_{m\to \infty}[R_{\vec{t}_{m}}]
=[R]
\in \overline{M_{g}}^{\rm DM}$ from our construction of 
$R_{\vec{t}_{m}}$), we have actually proved a stronger statement 
than the last paragraph of the statements of Theorem \ref{GH.curves}. 
\end{Rem}

\begin{Rem}
A while after the first version of this paper, we essentially gave another 
(logically independent) more moduli-theoritic proof of 
Theorem \ref{GH.curves} in the sequel 
\cite{Od.Ag} by using \cite{Wol}. 
Precisely speaking, Theorem \ref{GH.curves} follows from 
\cite[\S 3.2.1, Theorem 3.7 and its proof]{Od.Ag} which 
depends on \cite{Wol}. 
\end{Rem}

\begin{Rem}\label{ell.curve}
In the simpler case of $g=1$, i.e., elliptic curves case, 
we also have a similar phenomenon as discussed in the introduction of 
\cite{GW}. 
It can be regarded as the easiest prototypical example 
of the sequel paper \cite{Od.Ag} on the moduli spaces of 
principally polarized abelian varieties and also well-known to the 
experts of the Strominger-Yau-Zaslow mirror symmetry conjetures. 
Thus we give only brief description as an introduction to our sequels \cite{Od.Ag}, \cite{OO}. 

Suppose there is a sequence of elliptic curve $\{\mathbb{C}/
(\mathbb{Z}+\mathbb{Z}\tau_{i})\}_{i}$ 
where $\tau_{i}$ belongs to 
the standard fundamental domain $W$ of the upper half plane $\mathbb{H}$ modulo the modular group $\SL(2,\mathbb{Z})$, that is 
$$
W:=\{\tau\in \mathbb{H}\mid |{\rm Re}(\tau)|\leq 1, |\tau|\geq 1\}. 
$$
If ${\rm Im}(\tau_{i})$ does \textit{not} diverge, then after passing to a subsequence, they converge in the Gromov-Hausdorff sense to an elliptic 
curve. 
If ${\rm Im}(\tau_{i})$ diverges, then the Gromov-Hausdorff limit of a subsequence of $\biggl\{\biggl(R_{i},\dfrac{d_{\rm KE}}{{\rm diam}(d_{\rm KE})}\biggr)\biggr\}_{i=1,2,\cdots}$ is 
$S^{1}(\frac{1}{2\pi})$, the circle of radius $\frac{1}{2\pi}$. On the other hand, 
for a family of elliptic curves over the punctured disk, 
the compactified N\'{e}ron model after suitable base change is 
well-known to be $n$-gon with some $n\in \mathbb{Z}_{>0}$. Thus their dual graphs are topologically $S^{1}$, 
which is homeomorphic to the Gromov-Hausdorff limit discussed above. 

\end{Rem}

\begin{Rem}
For the case of curves with punctures (marked points), 
i.e., elements of $M_{g,n}$ with $n\ge 1$, as the natural hyperbolic metric has 
hyperbolic cusp singularities of \textit{infinite} diameters 
around the punctures, we have not been able to make a 
suitable formulation to study Gromov-Hausdorff collapses. 

 Professor Y-G.Oh kindly pointed out to me that a different but 
similar kind of ``graph-like thin" metrics also appear as 
``\textit{(general) minimal area metric}'' studied by Zwiebach and 
Wolf-Zwiebach (cf., e.g., \cite{Z}, \cite{WZ}) 
for constructing closed string field theory. 
The metrics are expected to be isometric to \textit{flat} 
semi-infinite cylinders around the punctures. 
The graph structure is regarded as a version of 
Feynman diagrams there. 

\end{Rem}

\begin{Rem}
Our Theorem \ref{GH.curves} 
suggests that the conjectures of Gross-Wilson \cite[\S6]{GW}, 
Kontsevich-Soibelman \cite{KS} and Gross-Siebert (cf., \cite{Gross}) 
on the correspondence of Gromov-Hausdorff limit and dual complex of 
degenerating \textit{Calabi-Yau manifolds} may well have an 
analogue in \textit{negative} Ricci curvature K\"ahler-Einstein 
case, i.e., those projective manifolds with ample canonical classes. 
\end{Rem}

Let us trace again the proof of our Theorem \ref{GH.curves} to see 
some analogy with the 
tropicalization of the Berkovich analytification \cite{ACP}. 
The one page arguments below does 
\textit{not} contain any substantially concrete results 
and rather we mean to give a re-interpretation of our Theorem \ref{GH.curves} 
and compare with \cite{ACP}. In our theorem \ref{GH.curves}, 
starting with an arbitrary sequence of compact hyperbolic surfaces, 
we took a nice subsequence which converges to a 
stable curve in the Deligne-Mumford compactification 
and also converging in the Gromov-Hausdorff sense (while fixing the diameter). 
Let us call such sequence of compact hyperbolic surfaces of genus $g(\ge 2)$ 
``strongly convergent sequence''. We denote the set of such strongly convergent sequences of 
compact hyperbolic Riemann surfaces as 
\footnote{Here, $\mathcal{S}$ stands for a sequence. }
$\mathcal{S}\mathcal{M}_{g}$. 

We denote by $S_{g}$ the 
moduli space of the underlying metric spaces of the 
metrized graphs which appear as the Gromov-Hausdorff 
limits of sequences of compact Riemann surfaces of genus $g(\ge2)$, as we showed in 
Theorem \ref{GH.curves}. We simply regard it as a 
Hausdorff topological space with a metric induced by the 
Gromov-Hausdorff distances among the graphs. Note that 
it is also compact by Theorem \ref{GH.curves} and the simple fact that 
$S_{g}$ is closed under the Gromov-Hausdorff convergence. 

Then what we have constructed in the proof of Theorem \ref{GH.curves} 
is the following two kinds of limiting maps 

\begin{equation}\label{map.r}
r\colon \mathcal{S}\mathcal{M}_{g}\rightarrow \overline{M_{g}}^{{\rm DM}} 
\end{equation}
\noindent
which maps $\{R_{i}\}$ to the limit (Deligne-Mumford) stable curve and 

\begin{equation}\label{map.t}
t\colon \mathcal{S}\mathcal{M}_{g}\rightarrow S_{g} 
\end{equation}
\noindent 
which maps $\{R_{i}\}$ to the Gromov-Hausdorff limit. 
Furthermore, we proved that $r$ and $t$ are compatible in the sense that 
the underlying graph of $t(\{R_{i}\})$ 
is a contraction of the dual graph of the limit stable curve $r(\{R_{i}\})$. 

On the other hand, in the recent paper \cite{ACP} by Abramovich-Caporaso-Payne, the following is proved. 

\begin{quotation}
\textit{
Fix an algebraically closed base field $k$ with trivial valuation. 
If we consider the Berkovich analytification $\overline{M_{g}}^{an}$ \cite{Berk1} of the 
Deligne-Mumford compactification $\overline{M_{g}}$, then 
the deformation retraction to the Berkovich skeleton \cite{Berk2} is 
the ``tropicalization" map towards the moduli of tropical curves of genus $g$. 
}
\end{quotation}

\noindent
Note that the Berkovich analytification parametrises stable curves over valuation fields which contains 
$k$ (with trivial valuation) and it can be regarded as (a subspace of) this as an ``algebro-geometric'' analogue of the set of strongly 
convergent sequence of compact Riemann surfaces $\mathcal{S}\mathcal{M}_{g}$. From this viewpoint, their tropicalization (deformation retract) 
is an analogue of our map $t$. 
The analogue of $r$ 
in the Berkovich geometric setting \cite{ACP} is the 
reduction map $\overline{M_{g}}^{an}\rightarrow \overline{M_{g}}^{\rm DM}$. 
Morally speaking, the anti-continuity of the reduction map \cite[(2.4)]{Berk1} 
is analogous to the phenomenon of reversing order of specialization and generization 
when going from the classical to the tropical world. 


\subsection{The construction of $\overline{M_{g}}^{\rm GH}$}

We define our \textit{Gromov-Hausdorff compactification} 
of the moduli space of curves, first set-theoretically as 
$$\overline{M_{g}}^{\rm GH}:=M_{g}\sqcup S_{g}. $$ 
Recall that we have defined $S_{g}$ in the previous subsection 
as the 
moduli space of the underlying metric spaces of the 
metrized graphs which appear as the Gromov-Hausdorff 
limits of sequences of compact Riemann surfaces of genus $g(\ge2)$. 
Then we put a topology on it, whose open basis consists of the following two kinds of 
subsets: 
\begin{enumerate}
\item open subsets of $M_{g}$ 
(with respect to the complex analytic topology) and 
\item the metrics balls with centers are in $S_{g}$. 
\end{enumerate}

What we mean by the metric ball,  
with its center $[G]\in S_{g}\subset \overline{M_{g}}^{\rm GH}$ 
($G$ is a metrized graph) and radius $r\in \mathbb{R}_{>0}$, is simply defined as 
$$
B([G],r):=\{[C]\in \overline{M_{g}}^{\rm GH} \mid d_{\rm GH}([C],[G])<r\}. 
$$
\noindent
The obtained topological space $\overline{M_{g}}^{\rm GH}$ is compact 
due to our Theorem \ref{GH.curves}. 
It also satisfies the Hausdorff separation axiom simply because the 
Gromov-Hausdorff limit as compact metric space is unique as general theory 
(cf., \cite{BBI}). 

Note that the complex conjugate $\iota \in {\rm Aut}(\mathbb{C}/\mathbb{R})$ 
reverses 
the natural orientation of the corresponding Riemann surface, 
which does not change it metric space strucuture. 
A subtle technical point here is that $\overline{M_{g}}^{\rm GH}$ is not 
exactly 
the metric completion with respect to the Gromov-Hausdorff topology, 
of the set 
of compact Riemann surfaces of genus $g$ by regarding the Riemann surfaces just as metric spaces. 
That is because it would discard the complex structures and ignore the effect of $\iota$ above (cf., e.g., \cite{Spo},\cite{OSS}). 

Recall that $S_{g}$ is defined as the moduli spaces of 
the underlying metric spaces of our limit metrized graphs 
as in Theorem \ref{GH.curves}.  
For each finite (metrized) graph $\Gamma$, let us denote the number of $1$-
valent vertices by $v_1(\Gamma)$ and denote the first betti number of $\Gamma$ by  $b_{1}(\Gamma)$. Then, more specifically and concretely, $S_{g}$ can be 
described as follows. 

\begin{Prop}\label{Sg.prop}
The metric spaces parametrized by $S_{g}$ can be characterized by a  
purely topological condition that 
the underlying topological spaces of the metrized graphs 
satisfy $v_1(\Gamma)+b_{1}(\Gamma)\leq g$. 
\end{Prop}

\noindent
Note there is a subtle distinction between the metrized graph and the underlying metric space, 
which is simply a $1$-dimensional CW complex with a metric. 
The reason 
is that the underlying metric space does \textit{not} see the $2$-valent vertices. It is also not 
enough to 
consider metrized graphs without $2$-valent vertices since a circle can not be obtained in that 
way. 

\begin{proof}
From Theorem \ref{GH.curves}, we only need to specify the class of 
dual graphs of stable curves with genus $g$. 

A stable curve $C$ of genus $g$ whose 
irreducible decomposition is $\cup_{i}C_{i}$ with dual graph $\Gamma$ 
satisfies 
\begin{equation}\label{wt.formula}
g=\sum_{i} g(C_{i}^{\nu})+b_1(\Gamma), 
\end{equation}
where $\nu$ denotes the normalization and $b_1$ denotes the first Betti number. 
From the stability condition, for each component $C_i$ which corresponds to a $1$-valent vertex 
of $\Gamma$, $g(C_i^{\nu})\geq 1$. This is essentially the only numerical stability 
condition. Thus we have $g=\sum_{i} g(C_{i}^{\nu})+b_1(\Gamma)\geq v_1(\Gamma)+b_1(\Gamma)$. 
Tracing back the above discussion, it is also easy to see that it is a sufficient condition as well. 
\end{proof}

\begin{Rem}
One remark, which the author hopes to be useful, is that in the above characterisation of metrized graphs which are parametrised in $S_{g}$, 
rather than putting the ``diameter $1$'' condition, it may be easier to impose 
that ``the sum of lengths of edges is $1$'' when we try to concretely describe the  structure of 
our compactifications. 
Note that these two moduli spaces are naturally homeomorphic, simply by rescaling the metrics on our metrized graphs. 
\end{Rem}


\section{Related moduli spaces and comparison}

In this section, we further study $\overline{M_g}^{\rm GH}$ somewhat indirectly by 
comparing with other moduli spaces in literatures, 
and also construct some variants of 
$\overline{M_g}^{\rm GH}$ on the way, including what we call 
tropical geometric compactifications and denote by $\overline{M_g}^{\rm T}$. 

\subsection{Comparison with tropical moduli spaces} 

Recently Brannetti-Melo-Viviani \cite{BMV} 
constructed moduli space $M_{g}^{tr}$ 
of tropical curves and Caporaso \cite{Cap} introduced its pointed versions $M_{g,n}^{trop}$ 
The moduli space $M_{g}^{tr}$ is similar to our boundary 
$S_g$ but there is an essential difference which is if the definition of tropical curves even encode \textit{genus} of each component or not. \cite{BMV} and \cite{Cap} called that data, weights. 

Similarly to what is done in  \cite{CV}, \cite{BMV}, \cite{Cap}, 
the combinatorial type of the underlying graph of a metrized graph 
gives a stratification on $S_{g}$ 
such that each strata is a finite quotient of a simplex. 
A basic property of our moduli space $S_{g}$ is the following. 

\begin{Prop}
The function $S_{g}\ni [\Gamma] \mapsto v_{1}(\Gamma)+b_{1}(\Gamma)$ is a lower semicontinuous function on $S_{g}$ with respect to the 
Gromov-Hausdorff topology which has been previously considered. 
\end{Prop}
\begin{proof}
The assersion follows easily from Theorem \ref{GH.curves} combined with the precompactness (Corollary \ref{Mg.precompact}) 
but let us give a more straightforward combinatorial proof. 

It is enough to see that if we contract one edge $e$, the function 
$v_{1}+b_{1}$ does not increase. 
If the edge $e$ is a loop, then the process decreases $b_{1}$ by $1$ and $v_{1}$ increases at most $1$. 
If the edge $e$ is not a loop, then the contraction does not change 
the homotopy type of the graph so that it keeps $b_{1}$ unchanged, 
and $v_{1}$ does not increase (it may decrease by $1$ or $2$). 

\end{proof}

Note that through the modular interpretations, 
there is a sequence of \textit{canonical} closed embeddings 
\begin{equation}\label{emb.Sg}
S_{g}\hookrightarrow S_{g+1}\hookrightarrow \cdots,  
\end{equation}
\noindent
while other compactifications of moduli of curves and 
the moduli of \textit{weighted} tropical curves by \cite{BMV}, \cite{Cap}, 
\cite{CHMR} 
do \textit{not} have this chain of canonical inclusions. 

Inside the moduli space $M_g^{tr}$ of (weighted) tropical curves 
in the sense of \cite{BMV}, 
let us consider the closed locus $S^{wt}_g$ which 
parametrizes those with the diameter $1$ 
(``\textit{wt}" of $S_g^{wt}$ stands for weights. ) 

\begin{Prop}\label{Sg.wt.Sg}
We have natural morphisms as follows.  
\begin{equation}\label{maps}
\partial M_g^{tr}:= M_g^{tr}\setminus \{ \text{a point with weight } g \} 
\cong S^{wt}_g\times \mathbb{R}_{>0} \twoheadrightarrow S^{wt}_g\twoheadrightarrow S_g. 
\end{equation}
$S_{g}$ has a finite stratification which satisfies that each strata is 
a finite group quotient of an open simplex and $S_{g}$ is 
``purely" $(3g-4)$ dimensional for each $g(\geq 2)$ in the sense that, 
if we denote the union of $(3g-4)$-dimensional strata 
as $S_{g}^{oo}\subset S_{g}$, then it is an open dense subset. 
In addition, the last morphism of \eqref{maps} is a proper map 
such that each fiber over $S_{g}^{\rm oo}$ is finite. 
\end{Prop}

\noindent
\begin{proof}
A tropical curve in the sense of \cite{BMV} has 
finite non-zero diameter unless it is a point, so that we get the first isomorphism. 
Secondly, starting from a tropical curve which is not topologically a point, 
just by forgetting the weights and 
the $2$-valent vertices, we get the underlying metric space of a 
metrized graph. 
It defines the last morphism $S^{wt}_{g}\twoheadrightarrow S_{g}$, 
which we denote as $r$. 
It follows straightforward from the topology on $S^{wt}_{g}$ in \cite{BMV} 
that this morphism is continuous and this is surjective by 
Proposition \ref{Sg.prop}. From the compactness of $S_{g}^{wt}$ and $S_g$, 
it follows automatically that the morphism is proper. Note that  
for any point $p$ in $S_{g}^{wt}$ which has $2$-valent vertices, 
$r^{-1}(r(p))$ is non-finite. 
It is because 
that for each metric space $X$ corresponding to a point in $S_g$, if it is  underlying 
metric space of certain weighted tropical curve (weighted metrized graph) $\Gamma$ parametrized in $S_g^{wt}$, 
once we know the locations of vertices in $X$, there are only finite choices of 
$\Gamma$ which corresponds to the decomposition of $g-b_1(X)$ into non-negative 
integer weights attached to the vertices. 

It is easy to see that $S_{g}$ has a natural finite stratification 
by the homeomorphic class of the underlying graphs. 
Each strata can be seen as the moduli of 
metrized graphs with the same underlying graph, 
with the sum of the length of edges are $1$ by rescaling the metrics. 
Hence it is homeomorphic to 
the quotient of an open simplex with respect to a linear action of 
a finite group, which is the automorphism group of each graph. 
Next we proceed to the proof of the fact that 
$S_{g}$ is purely $3g-4$-dimensional as in the statement of 
Proposition \ref{Sg.wt.Sg}. 
Indeed, for any given (underlying metric space of) 
a metrized graph $\Gamma$ of the diameter $1$ which satisfies 
$v_1(\Gamma)+b_1(\Gamma)<g$, by attaching small 
circles or short edges and rescaling, 
the corresponding point $[\Gamma]\in S_g$ can be 
easily perturbed to a point inside the strata 
with  $v_1+b_1=g$. The strata can be easily checked to have 
dimension $3g-4$, as $3g-3$ is the number of edges inside $\Gamma$ 
following elementary graph theory. This fact is also well known in the algebro-geometric field 
of study of the so-called Mumford curves. 
Thus, the 
union $S_{g}^{oo}$ of $(3g-4)$-dimensional cells form open dense subset. 
For each $p\in S_{g}^{oo}$, 
the $r$-fiber $r^{-1}(r(p))$ is finite since the point in the fiber 
the corresponding tropical curve does not have any positive weights 
on vertices because of the 
formula (\ref{wt.formula}). We complete the proof of Proposition \ref{Sg.wt.Sg}. 
\end{proof}

\subsection{Construction of $\overline{M_{g}}^{\rm T}$}
It is possible to modify 
our construction of $\overline{M_{g}}^{\rm GH}$ to make more 
compactibility 
with the above ``\textit{weighted} tropical moduli spaces" of \cite{BMV}, \cite{Cap}, 
\cite{CHMR}. 
That is, for a collapsing sequence of genus $g$ compact Riemann surfaces 
as in Theorem \ref{GH.curves}, we can 
encode the information of the genera of the irreducible components of 
the limiting stable curves on the limiting graph. More precisely speaking, 
first we consider the set 
$$
\overline{M_{g}}^{\rm T}:=M_{g}\sqcup S_{g}^{wt}, 
$$
on which we put a topology as follows. 
A subset $C$ of $\overline{M_{g}}^{\rm T}$ is \textit{closed} if and only if 
\begin{itemize}
\item $C\cap S_{g}^{wt}$ is closed in $S_{g}^{wt}$ and 
\item any Gromov-Hausdorff collapsed graphs of compact Riemann surfaces 
which are in $C\cap M_{g}$, attached with the genera of 
components of the normalization of the limit stable curve in \cite{DM} sense, 
which we suppose to exist, 
is actually in $C\cap S_{g}^{wt}$. 
\end{itemize}
The compactness, the Hausdorff property of $\overline{M_g}^{\rm T}$, 
and the fact that $M_{g}$ is open and dense inside $\overline{M_g}^{\rm T}$ 
all follow straightforwardly from our 
Theorem \ref{GH.curves} and its proof. 
We would like to call this compactification $\overline{M_g}^{\rm T}$ of $M_{g}$ as 
the \textit{tropical geometric compactification} of $M_{g}$. 

From the construction we have a natural continuous surjective map 
$$\overline{M_{g}}^{\rm T}\twoheadrightarrow \overline{M_{g}}^{\rm GH},$$ 
which restricts to the identity map on the open subset $M_{g}$.


\subsection{Finite join $\overline{M_{\leq g}}^{\rm GH}$ and infinite join $\overline{M_{\infty}}^{\rm GH}$}\label{Mg.join}

An interesting phenomenon is that, as the following definitions show, our Gromov-Hausdorff compactification 
$\overline{M_{g}}^{\rm GH}$ naturally patches together for different $g$ thanks to the 
sequence of the 
canonical inclusions $(\ref{emb.Sg})$ of $S_{g}$. 

\begin{Def}

The \textit{finite joins} of our Gromov-Hausdorff compactifications are 
defined inductively as topological spaces 

$$\overline{M_{\leq 0}}^{\rm GH}:=\overline{M_{0}}^{\rm GH}=\{\text{ Riemann sphere }\mathbb{CP}^{1}\} \text{ (singleton)}, $$

$$\overline{M_{\leq 1}}^{\rm GH}:=\overline{M_{1}}^{\rm GH}:=M_{1}\sqcup \{S^{1}\bigl(\frac{1}{2\pi}\bigr)\}(=\overline{A_{1}}^{\rm T} \text{ in the next section })$$ 
\begin{center} 
(one point compactification) 
\end{center} and for $g\ge 2$ as

$$\overline{M_{\leq g}}^{\rm GH}:=
(\overline{M_{\leq (g-1)}}^{\rm GH}\cup\overline{M_{g}}^{\rm GH})/\sim, $$
where the equivalence relation $\sim$ 
is simply the identification of the closed subset 
$S_{g-1}\subset S_{g}$ and another closed subset 
$S_{g-1}\subset \overline{M_{\leq (g-1)}}^{\rm GH}$. 
From the definition, we have natural inclusion relations 
$$\cdots\overline{M_{\leq (g-1)}}^{\rm GH}\subset 
\overline{M_{\leq g}}^{\rm GH}\cdots.$$
Then we set 
$$
\overline{M_{\infty}}^{\rm GH}:=
\varinjlim_{g} \overline{M_{\leq g}}^{\rm GH}=\cup_{g}\overline{M_{\leq g}}^{\rm GH},  
$$
and call it \textit{infinite join} of our Gromov-Hausdorff compactifications. 
\end{Def}

The boundary of $\overline{M_{\infty}}^{\rm GH}$ by which we mean the natural subset 
$\cup_{g}
(\partial \overline{M_{g}}^{\rm GH}=S_{g})$, should be regarded as a tropical version of 
the space\footnote{They call it ``universal moduli spaces''} 
``$M_{\infty}$'' 
introduced and studied recently 
by Ji-Jost \cite{JJ}. 

Also note that $\overline{M_{\infty}}^{\rm GH}$ 
is connected and all our Gromov-Hausdorff compactification 
$\overline{M_{g}}^{\rm GH}$ is inside this infinite join. 

\subsection{Comparison with the Outer spaces}

There is a classical theory of the \textit{outer space} $X_{n}$ by Culler-Vogtman 
\cite{CV}, 
which is an analogue of the Teichmuller space for metrized graphs. There, the 
analogous discrete group to the mapping class group 
is the outer automorphism group ${\rm Out}(F_{n})$ of the free group $F_{n}$ with 
rank $n$. 
From now on, we use $g$ instead of their $n$ to unify our notation. 

Recall that the quotient $X_{g}/{\rm Out}(F_{g})$ 
parametrizes graphs $\Gamma$ with $b_{1}(\Gamma)=g$ with $v_{1}(\Gamma)=0$. 

We introduce another moduli space of graphs as a subset of $S_{g}$ (with 
the induced topology) 
as
$$
S_{g}^{o}:=\{ \Gamma  \in S_{g} \mid v_{1}(\Gamma)+b_{1}(\Gamma)=g\}\subset S_{g}. 
$$

\noindent 
It is simply the complement of $S_{g-1}\subset S_{g}$ by the definition. 
The following proposition essentially goes back to \cite{CMV}. 
\begin{Prop}
There is a canonical cellular open embedding $X_{g}/{\rm Out(F_{g})}\hookrightarrow S_{g}^{o}(\subset S^{g})$. 
The image of $X_{g}/{\rm Out(F_{g})}$ is open dense in $S_{g}$ (thus so is $S_{g}^{o}$). 
\end{Prop}

\begin{proof}
First of all, it follows from the lower semicontinuity of 
the first Betti number of metrized graphs $b_1(\Gamma)$ that 
$X_{g}/{\rm Out}(F_{g})$ is an 
open subset of $S_{g}^{o}$. 
For each $\Gamma\in S_{g}^{o}$ with $v_{1}(\Gamma)+b_{1}(\Gamma)=g$ and 
$0< \epsilon\ll 1$, we define graph(s) $\phi_{\epsilon}(\Gamma)$ as follows. For each leave $vw$ where $v$ is a $1$-valent vertix, 
we put a small loop of length $\epsilon l(vw)$. Doing the same for all edges and rescale the metric on whole graph to make its diameter $1$, 
we get a metrized graph 
which we denote as $\phi_{\epsilon}(\Gamma)$. This construction 
naturally defines a perturbation of elements of $S_{g}^{o}$ to those of 
$X_{g}/{\rm Out}(F_{g})$. 
The fact that all of these are unions of relative interiors of 
the cells with respect to that CW complex structure follow straightforward from the definitions. 

We also need to prove $S_{g}^{o}$ is dense inside $S_{g}$. 
We provide an elementary proof for convenience. Let us analyze the neighborhood of $\Gamma\in S_{g-1}\subset S_{g}$. Starting 
from any such $\Gamma$ with a point $p\in \Gamma$, 
we can similarly consider $\Gamma$'s deformation $\psi_{t}(\Gamma)\in X_{g}/{\rm Out}(F_{g})$ for $t>0$, for example, as follows. 
Set $v_{1}(\Gamma)+b_{1}(\Gamma)=g-d$. Taking a point $p$, 
we define $\psi_{t}(\Gamma)$ as a union of $\Gamma$ and a bouquet i.e., the union of $d$ length $t$ loops which passes through $p$. 
Thus in particular $X_{g}/{\rm Out}(F_{g})$ is open and dense in $S_{g}$ 
and hence so is $S_{g}^{o}$ as well. 
\end{proof}

\subsection{Note added: on L.Lang's ``tropical convergence"}

We end this section with the following notes added, about the relation with 
\cite{LL} 
which was kindly taught by its author L.Lang in June of 2015. I appreciate 
him for informing it. 

\begin{Rem}\label{LL}
L.~Lang defined ``tropical convergence" of compact Riemann surfaces 
to metrized graphs as the convergence of the ratios of the lengths of 
shrinking geodesics, which 
represent vanishing cycles, in his \cite[Definition 1.1]{LL}. 
As also written in \cite[v2, \S1.3]{LL}, 
that notion of convergence is \textit{not} equivalent to ours, i.e. Gromov-Hausdorff convergence 
of hyperbolic metrics. See more details on 
the original paper \cite{LL}. The author also gives more detailed arguments in \cite[\S 3]{Od.Ag}. 
\end{Rem}


\section{Investigating topology}

We would like to make the first step of investigation of the topology of our 
compactifications and their boundaries. 


First, we recall the fact that the moduli space of smooth projective curves has vanishing higher homology groups, 
proved by J.~Harer \cite{Har}. 
His proof shows the existence of a deformation retract via the 
cell complex structure of the Teichmuller space (the so called ``arc complex''). 

\begin{Thm}[{\cite[Theorem 4.1]{Har}}]\label{Harer}
For $g\geq 2$ and $i>4g-5$, we have 
$$H_{i}(M_{g};\mathbb{Q})=0 \text{ and } H^{i}(M_{g};\mathbb{Q})=0. $$ 
So combined with the Poincar\'e-Lefschetz duality for orbifold, 
we get that for $i\leq 2g-2$  
$$
H^{i}_{\rm c}(M_{g};\mathbb{Q})=0 \text{ and } 
H_{i}^{\rm BM}(M_{g};\mathbb{Q})=0, 
$$
\noindent
where $H^{i}_{\rm c}$ denotes the cohomology group with compact supports and $H_{i}^{\rm BM}$ denotes 
the Borel-Moore homology group. 

\end{Thm}

The above theorem \ref{Harer} has the following consequence. 

\begin{Cor}
For $i< 2g-2$, we have 
$$H^{i}(\overline{M_{g}}^{\rm T};\mathbb{Q})=H^{i}(S_{g};\mathbb{Q}),  $$
$$
H_{i}(\overline{M_{g}}^{\rm T};\mathbb{Q})=H_{i}(S_{g};\mathbb{Q}). 
$$
\end{Cor}

\begin{proof}
It follows simply from the exact sequences of compactly supported cohomology groups or the Borel-Moore homology groups. 
\end{proof}


Thus the study of homology and cohomology of our Gromov-Hausdorff 
compactification is 
reduced to that of the boundary for a specific range of degrees. 
Motivated by it, let us study the topology 
\footnote{A while after the appearance of the first version of this paper as 
arXiv:1406.7772, Chan-Galatius-Payne \cite{CGP16} appears which systematically 
studies the topology of the moduli of weighted metrized graphs \textit{
with $n(>0)$-marked points} i.e. the ``log version" 
of $S_g^{wT}$. }
of our boundary $S_{g}$. 
First, we sketch the following cases of small $g$. 
\begin{Ex}
$S_{1}$ is a point which stands for the circle of length $1$. $S_{2}$ is a 
two $2$-simplices (triangles) patched together along one of their edges for each. The latter is 
also homeomorphic to a $2$-simplex again so $S_{1}$ and $S_{2}$ are both contractable. 
\end{Ex}

Since the open dense locus $S_{g}^{o}$ of $S_{g}$ is a rational 
classifying space of ${\rm Out}(F_{g})$ as known to \cite{CV}, 
it has in general highly nontrivial topology. Indeed 
its cohomology is those of ${\rm Out}(F_{g})$ (cf., e.g., \cite{EVHS} for non-vanishing cohomology for $g=5$ case), 
we expect interesting topological structure on $S_{g}$ for large $g$. 

We define $$S_{\infty}:=\varinjlim S_{g}=\cup_{g} S_{g},$$ 
the injective limit 
with respect to the canonical embeddings $S_{g-1}\hookrightarrow S_{g}\hookrightarrow S_{g+1}\cdots$ (cf., (\ref{emb.Sg})). 
Then, while we expect that each $S_{g}$ has highly nontrivial topologies in general, 
we observe the following. 

\begin{Thm}\label{curve.bd.hm}
The topological space $S_{\infty}$ is contractible. 
In particular, for any $k\ge 0$, $\displaystyle \varinjlim_{g} H_{k}(S_{g};\mathbb{Q})=0$. 
\end{Thm}

\begin{proof}
Consider the cone of $S_{g}$, i.e.,  
$CS_{g}:=(S_{g}\times [0,1])/(S_{g}\times \{1\})$. 
It is enough to construct a series of continuous maps $\{\phi_{g}\colon CS_{g}\rightarrow S_{\infty}\}_{g\ge 2}$ which satisfies 
\begin{enumerate}
\item \label{cond1} $\phi_{g}
(S_{g}\times \{1\})=\{\text{the segment of length}1\}$, 
\item \label{cond2} $\phi_{g}|_{S_{g}\times \{0\}}={\rm id}|_{S_g}$, 
\item \label{cond3} and $\phi_{g+1}|_{CS_{g}}=\phi_{g}$. 
\end{enumerate}
Indeed, from the third condition, they glue together to form a continuous map 
$$\phi_{\infty}\colon CS_{\infty}\to S_{\infty}$$ and this gives a deformation 
retract of $S_{\infty}$ into a point of $S_{\infty}$ which corresponds to 
the segment of length $1$. 

We construct the map $\phi_{g}$ by the following three steps.

\begin{Step}[Adding vertices]\label{Step1}
First we construct $\phi_{g}|_{S_{g}\times [0,\frac{1}{3}]}$. 
For any $(\Gamma, t)\in S_{g}\times [0,\frac{1}{3}]$, 
suppose the set of vertices of $\Gamma$ is 
$V(\Gamma)=\{v_{1},\cdots,v_{m}\}$ and the set of edges 
is $E(\Gamma)=\{e_{1},\cdots,e_{n}\}$. 
We define a new metrized graph 
$\psi_{g}(\Gamma,t)$ for $t\in (0,\frac{1}{3}]$ 
by setting the vertices set 
as $\{v_{1},\cdots,v_{m}\}\sqcup \{w_{1},\cdots,w_{m}\}$ 
and define the set of edges and their lengths as follows. 
The set of edges of $\psi_{g}(\Gamma,t)$ is 
$E(\Gamma)\sqcup \{\overline{v_{i}w_{i}}\mid 1\le i\le m\}$. 
We call an edge in $E(\Gamma)\subset E(\psi_{g}(\Gamma,t))$ 
as \textit{old edge} in this proof, while  
the edges of the form $\overline{v_{i}w_{i}}$ will be called 
\textit{new edges}. 
We put their length $l(v_{i}w_{i})=t$ while we keep the length of 
old edges as the same as $\Gamma$. Then we rescale the length of all edges of 
$\psi_{g}(\Gamma,t) (0< t\le \frac{1}{3})$ to make the diameter $1$ and denote the obtained 
metrized graph as 
$\phi_{g}(\Gamma,t)$. Note that the image of $\phi_{g}|_{S_{g}\times [0,\frac{1}
{3}]}$ is a priori \textit{not} inside $S_{g}$. 
This $\phi_{g}|_{S_{g}\times [0,\frac{1}{3}]}$ is continuous from the construction. 
\end{Step}

\begin{Step}[Contraction of old edges]\label{Step2}
Our next step is the construction of $\phi_{g}|_{S_{g}\times [\frac{1}{3},\frac{2}{3}]}$. 
Roughly speaking, in this step of $t$ increasing from $\frac{1}{3}$ to $\frac{2}
{3}$, we gradually contract the old edges i.e., those which belong to $E(\Gamma)$. 
We make this rigorous as follows. 

First, as in Step \ref{Step1}, 
we construct $\psi_{g}(\Gamma,t)$ for $t\in [\frac{1}{3},\frac{2}{3}]$ 
by setting its vertices set and edges set as 
\begin{align*}
V(\psi_{g}(\Gamma,t))&:=V(\phi_{g}(\Gamma,\frac{1}{3}))\\ 
&=\{v_{1},\cdots,v_{m}\}\sqcup \{w_{1},\cdots,w_{m}\} \text{ for } t\in 
\bigl[\frac{1}{3},\frac{2}{3}\bigr), \\ 
V(\psi_{g}(\Gamma,t))&:=\{v\}
\sqcup \{w_{1},\cdots,w_{m}\}\text{ for }t=\frac{2}{3}, \\ 
E(\psi_{g}(\Gamma,t))&:= E(\phi_{g}(\Gamma,\frac{1}{3}))\\
&=E(\Gamma)\sqcup \{\overline{v_{i}w_{i}}\mid 1\le i\le n\}
\text{ for }t\in \bigl[\frac{1}{3},\frac{2}{3}\bigr), \\ 
E(\psi_{g}(\Gamma,t))&:= 
\{\overline{vw_{i}}\mid 1\le i\le n\}
\text{ for }t=\frac{2}{3}. 
\end{align*}

Then we put the metrics on the edges of $\phi_{g}(\Gamma,t)$ as follows. 
\footnote{
The notation of the following is that the 
length of edge $l$ in a graph $G$ is denoted as ${\rm length}(l,G)$. }
\begin{align*}
{\rm length}(\overline{v_{i}w_{i}}; \phi_{g}(\Gamma,t))&:=\frac{1}{3},\\ 
{\rm length}(\overline{v_{i}v_{j}}; \phi_{g}(\Gamma,t))&:=(2-3t)
{\rm length}(\overline{v_{i}v_{j}}; \Gamma). 
\end{align*}
The above construction of $\psi_{g}(\Gamma,t)$ realizes shrink of old edges in 
$\phi_{g}(\Gamma,\frac{1}{3})$. Then finally we define the metrized graph 
$\phi_{g}(\Gamma,t)$ as rescale of $\psi_{g}(\Gamma,t)$ with the 
diameter $1$. 

From the constrution, the continuity of $\psi_{g}|_{S_{g}\times [\frac{1}{3},
\frac{2}{3}]}$ and $\phi_{g}|_{S_{g}\times [\frac{1}{3},
\frac{2}{3}]}$ are obvious. 
The limit graph $\phi_{g}|_{t=\frac{2}{3}}$ is a tree whose shape looks 
like ``$*$'' 
which we may call ``star-graphs"
Let us call this type of tree ``$*$-type'' with $n(=\# E(\Gamma))$ 
leaves. 

\end{Step}
\begin{Step}[Deforming to the unit segment]
The final step is the construction of $\phi_{g}|_{S_{g}\times [\frac{2}{3},1]}$. 
The moduli space of $*$-type trees $\Gamma$ (as we discussed above in Step \ref{Step2}) with $n$ leaves of diameter $1$ is homeomorphic to the moduli 
space of those whose sum of lengths of edges is $1$, simply by rescaling. And the latter is the simplex 
$$
\Delta_{n}:=\{(x_{1},\cdots,x_{n})\mid 0\leq x_{1}\leq x_{2}\leq \cdots x_{n}\leq 1 , \sum^{n}_{i=1} x_{i}=1\}. 
$$ 
The contractability of the simplex above ensures, or we can directly see that 
there is a deformation retract of each $\Gamma\in \Delta_{n}$ 
to the interval $[0,1]$. This gives $\phi_{g}|_{S_g\times [\frac{2}{3},1]}$. 
\end{Step}
The desired properties \eqref{cond1}, \eqref{cond2}, \eqref{cond3} are 
all straightforward from the construction. 
We complete the proof of Theorem \ref{curve.bd.hm}. 
\end{proof}




\vspace{5mm} \footnotesize \noindent
Contact: {\tt yodaka[at]math.kyoto-u.ac.jp} \\
Department of Mathematics, Kyoto University, Kyoto 606-8285. JAPAN \\


\begin{thebibliography}{FGA}

\bibitem[ACP]{ACP}
D.~Abramovich, L.~Caporaso, S.~Payne, 
The tropicalization of the moduli space of curves, 
Ann.\ Sc.\ de l'ENS vol.\ 48 (2015), 765-809. 

\bibitem[BBI]{BBI}
D.~Burago, Y.~Burago, S.~Ivanov, 
A course in metric geometry, 
Graduate Studies in Mathematics, AMS, 
Volume: 33; (2001). 

\bibitem[Berk1]{Berk1}
V.~Berkovich, 
Spectral theory and analytic geometry over non-archimedean fields, 
Mathematical surveys and monographs, no.\ 33, American Math Society (1990). 

\bibitem[Berk2]{Berk2}
V.~Berkovich, 
Smooth $p$-adic analytic spaces are locally contractible, 
Invent.\ Math.\ vol.\ 137, no.\ 1, pp. 1-84 (1999). 

\bibitem[Bers]{Bers}
L.~Bers, 
An inequality for Riemann surfaces, Differential geometry and complex analysis, 
Springer, Berlin, pp. 87-93 (1985). 

\bibitem[BG]{BG}
R.~Berman, H.~Guenancia, 
K\"ahler-Einstein metrics on stable varieties and log canonical pairs, 
Geometric and Functional Analysis 24 (6), 1683-1730 (2014)

\bibitem[Bor]{Bor}
A.~Borel, 
Stable real cohomology of arithmetic groups, 
Ann.\ Sci.\ de L'\'{E}.\ N.\ S., vol.\ 4 pp. 235-272 (1974). 

\bibitem[BMV]{BMV}
S.~Brannetti, M.~Melo, F.~Viviani, 
On the tropical Torelli map, Adv.\ in Math.\ 
vol.\ 226 pp. 2546-2586 (2011). 

\bibitem[Cap]{Cap}
L.\ Caporaso, 
Algebraic and tropical curves: comparing their moduli spaces, 
Handbook of moduli. Vol. I, Adv.\ Lect.\ Math.\ (ALM), vol.\ 24, 
Int.\ Press, Somerville , MA, (2013). 

\bibitem[CGP]{CGP16}
M.~Chan, S.~Galatius, S.~Payne, 
The tropicalization of the moduli space of curves II: Topology and applications
arXiv:1604.03176. 

\bibitem[CHMR]{CHMR}
R.~Cavalieri, S.~Hampe, H.~Markwig, D.~Ranganathan, 
Moduli spaces of rational weighted stable curves and tropical geometry
arXiv:1404.7426 (2014). 

\bibitem[CMV]{CMV}
M.~Chan, M.~Melo, F.~Viviani, 
Tropical Teichmull\"er and Siegel spaces, 
Algebraic and Combinatorial Aspects of Tropical Geometry, Contemp. Math. 589 (2013), 
45--85. 

\bibitem[CV]{CV}
M.Culler, K.Vogtmann, 
Moduli of graphs and automorphisms of free groups, 
Invent.\ Math.\ vol.\ 84 no.1, pp. 91--119. (1986)

\bibitem[DM]{DM}
P.~Deligne, D.~Mumford, 
The irreducibility of moduli of curves, 
Publ.\ I.\ H.\ E.\ S.\ (1969). 

\bibitem[DS]{DS}
S.~Donaldson, S.~Sun, 
Gromov-Hausdorff limits of Kahler manifolds and algebraic geometry, 
Acta Math (2014). 

\bibitem[EVHS]{EVHS}
P.~Elbaz-Vincent, G.~Herbert, C.~Soul\'e, 
Quelques calculs de la cohomologie de ${\rm GL}_{N}(\mathbb{Z})$ et de la 
K-th\'eorie de $\mathbb{Z}$, 
C.\ R.\ Math.\ Acad.\ Sci.\ Paris vol.\ \textbf{335}, no.\ 4, pp.321-324 (2002). 

\bibitem[Gie]{Gie}
D.~Gieseker, Lectures on moduli of curves. Tata Institute of Fundamental Research Lectures on Mathematics and Physics, 69. Springer-Verlag, (1982). iii+99 pp.

\bibitem[Grom]{Grom}
M.~Gromov, 
Structures m\'{e}triques pour les 
vari\'{e}t\'{e}s riemanniennes, Textes Math\'{e}matiques, Paris, no.\ 1, pp. 1-120 (1981). 

\bibitem[Gross]{Gross} M.~Gross, 
Mirror Symmetry and the Strominger-Yau-Zaslow conjecture, 
Current developments in mathematics 2012, Int.\ Press, Somerville (2013), 
pp. 133-191. 

\bibitem[GW]{GW}
M.~Gross, P.~M.~H. Wilson, 
Large complex structure limits of K3 surfaces, 
J.\ Differential Geom.\ vol.\ 55, No.\ 3, pp.475--546 (2000). 

\bibitem[GTZ16]{GTZ}
M.~Gross, V.~Tosatti, Y.~Zhang, 
Gromov-Hausdorff collapsing of Calabi-Yau manifolds, 
Comm.\ in Anal.\ and Geom., vol.\ 24 (2016), 93-113. 

\bibitem[Ham]{Ham11}
U.~Hamenst\"adt,
Teichmuller theory, 
IAS/Park city Mathematics series vol.\ 20 (2011). 
B.~Farb, R.~Hain, E.~Looijenga ed. 

\bibitem[Har]{Har} 
J.~Harer, 
The virtual cohomological dimension 
of the mapping class group of an oriented surface, 
Invent.\ Math.\ vol.\ 84, pp.157-176 (1986). 

\bibitem[IT]{IT} 
Y.~Imayoshi, M.~Taniguchi, 
An introduction to Teichm\"{u}ller spaces, Springer-Verlag (1992)

\bibitem[JJ]{JJ}
L.~Ji, J.~Jost, 
Universal moduli spaces of Riemann surfaces, 
Jour.\ of Geom.\ Physics
vol.\ 114, (2017), pp. 124-137. 

\bibitem[Ke]{Ke}
L.~Keen, 
Collars on Riemann surfaces, 
Discontinuous Groups and Riemann Surfaces, 
Princeton University Press, pp. 263-268 (1974). 

\bibitem[Kol]{Kol}
J.~Koll\'ar, 
Moduli of varieties of general type, 
in Handbook of Moduli, vol. II. Advanced Lectures in Mathematics, vol.\ 25 
(International Press, Boston, 2013), pp. 131-167. 

\bibitem[KM]{KM}
(I) F.~F.~Knudsen, D.~Mumford, 
The projectivity of the moduli space of stable curves. I. Preliminaries on ``det'' and ``Div''. Math. Scand. 39 (1976), no. 1, 19-55, 

(II) F.~Knudsen, The projectivity of the moduli space of stable curves. II. The stacks $M_{g,n}$. Math. Scand. 52 (1983), no. 2, 161-199. 

(III) F.~Knudsen, The projectivity of the moduli space of stable curves. III. The line bundles on $M_{g,n}$, 
and a proof of the projectivity of $M_{g,n}$ in characteristic $0$. 
Math. Scand. 52 (1983), no. 2, 200-212. 

\bibitem[KS]{KS} M.~Kontsevich, Y.~Soibelman, 
Affine structures and non-archimedian geometry, 
The Unity of Mathematics
Progress in Mathematics vol.\ 244, pp 321-385 (2006). 

\bibitem[LL]{LL}
L.~Lang, 
Harmonic tropical curves, 
arXiv:1501.07121v2. 

\bibitem[LW]{LW}
J.~Li, X.~Wang, 
Hilbert-Mumford criterion for nodal curves, 
Compositio Math (2015). 

\bibitem[LWX]{LWX}
C.~Li, X.~Wang, C.~Xu, 
Degeneration of Fano K\"ahler-Einstein varieties, 
arXiv:1411.0761v2. 

\bibitem[MV]{MV}
B.~Mirzaii, W.~Van der Kallen, 
Homology stability for symplectic groups, 
arXiv:0110163 (2001). 

\bibitem[MZ]{MZ}
G.~Mikhalkin, I.~Zharkov, 
Tropical curves, their Jacobians and Theta functions, 
Contemporary Mathematics vol.\ 465, 
Proceedings of the Interna- tional Conference on Curves and Abelian Varieties in honor of Roy Smith's 65th birthday, 
pp. 203-231 (2007). 

\bibitem[Mum1]{Mum65}
D.~Mumford, 
Geometric Invariant Theory, 
Ergebnisse Math, Springer-Verlag (1965). 

\bibitem[Mum2]{Mum77}
D.~Mumford, 
Stability of projective varieties, 
Enseignement Math. (2) 23 (1977), no. 1-2, 39-110. 

\bibitem[Od1]{Od1}
Y.~Odaka, 
The GIT stability of polarized varieties via Discrepancy, 
Ann.\ of Math (2013). 

\bibitem[Od2]{Od2}
Y.~Odaka, 
A generalization of Ross-Thomas slope theory, 
Osaka J.\ Math.\ (2013) 

\bibitem[Od3]{Od3}
Y.~Odaka, 
The Calabi conjecture and K-stability, 
I.\ M.\ R.\ N.\  vol.\ 2012, No. 10, pp. 2272-2288 (2012). 

\bibitem[Od4]{Od4}
Y.~Odaka, 
On the moduli of K\"ahler-Einstein Fano manifolds, 
Proceeding of Kinosaki algebraic geometry symposium $2013$. 
(arXiv:1211.4833 v4) 

\bibitem[Od5]{Od5}
Y.~Odaka, 
Compact moduli spaces of K\"ahler-Einstein Fano manifolds, 
Publ. R.\ I.\ M.\ S (2015). 

\bibitem[Od6]{Od.Ag}
Y.~Odaka, 
Tropical geometric compactification of moduli, II 
- $A_{g}$ case and algebraic limits -, 
I.M.R.N. 2018 (It includes a developed version of the \textit{latter} half of 
arXiv:1406.7772v1. )

\bibitem[OO]{OO}
Y.~Odaka, Y.~Oshima, 
in preparation. 

\bibitem[OSS]{OSS}
Y.~Odaka, C.~Spotti, S.~Sun, 
Compact moduli of Del Pezzo surfaces and K\"ahler-Einstein metrics, 
J.\ Diff.\ Geom.\ Volume 102, Number 1 (2016), 127-172.
arXiv:1210.0858. 


\bibitem[Shi]{Shi}

T.~Shioya, The limit spaces of two dimensional manifolds 
with uniformly bounded integral curvature, Trans.\ A.M.S.\ 
vol.\ 351, No.5, pp.1765-1801 (1999)

\bibitem[Spo]{Spo}
C.~Spotti, 
Degenerations of K\"{a}hler-Einstein Fano manifolds, 
Imperial College Ph.\ D Thesis (2012). 

\bibitem[SSY]{SSY}
C.~Spotti, S.~Sun, C.~Yao, 
Existence and deformations of Kahler-Einstein metrics on smoothable 
$\mathbb{Q}$-Fano varieties, 
Duke Math. J. 165, no. 16 (2016). 

\bibitem[SYZ]{SYZ} A.~Strominger, S.~T.~Yau, E.~Zaslow, 
Mirror symmetry is T-duality, 
Nuclear Physics B vol.\ 479 pp.243-259 (1996). 

\bibitem[Tev]{Tev}
J.~Tevelev, 
Compactifications of subvarieties of tori, 
Amer.\ J.\ of Math.\ vol.\ 129, pp. 1087-1104 (2007). 

\bibitem[Wol]{Wol}
S.~Wolpert, 
The hyperbolic metric and the geometry of the universal curve, 
J.\ Diff.\ Geom.\ vol.\ 31 no. 2, pp. 417-472 (1990). 

\bibitem[WZ]{WZ}
M.~Wolf, B.~Zweibach, 
The plumbing of minimal area surfaces, 
Journal of Geom.\ and Physics 15 (1994), 23-56. 

\bibitem[Z]{Z}
B.~Zwiebach, 
How covariant closed string theory solves 
a minimal area problem, Commun.\ Math.\ Phys.\ (1991). 


\end{thebibliography}
\end{document}